\documentclass[12pt]{amsart}
\usepackage[applemac]{inputenc}

\usepackage[colorlinks=true,linkcolor=magenta,citecolor=blue,urlcolor=blue,hypertexnames=false,linktocpage]{hyperref}
\usepackage{bookmark}
\usepackage{amsthm,thmtools,amssymb,amsmath,amscd,amsfonts}
\usepackage{cleveref}
\usepackage{fancyhdr}
\usepackage{esint}
\usepackage{enumerate}

\usepackage{pictexwd,dcpic}

\newtheorem{theorem}{Theorem}[section]

\newtheorem{Assum}{Assumption}[section]
\newtheorem{lemma}[theorem]{Lemma}
\newtheorem{proposition}[theorem]{Proposition}
\newtheorem{corollary}[theorem]{Corollary}
\theoremstyle{definition}
\newtheorem{definition}[theorem]{Definition}

\theoremstyle{remark}
\newtheorem{remark}[theorem]{Remark}

\numberwithin{equation}{section}

\newcommand{\n}{\bar{\nabla}}
\newcommand{\ep}{\varepsilon}

\newcommand{\Sp}{\mathbb{S}^n}
\newcommand{\Euc}{\mathbb{R}^{n+1}}

\begin{document}

\title[]
 {Orlicz--Minkowski flows}
\author[P. Bryan, M. N. Ivaki, J. Scheuer]{Paul Bryan, Mohammad N. Ivaki, Julian Scheuer}
\dedicatory{}
\date{\today}
\subjclass[2010]{}
\begin{abstract}
We study the long-time existence and behavior for a class of anisotropic non-homogeneous Gauss curvature flows whose stationary solutions, if exist, solve the regular Orlicz--Minkowski problems. As an application, we obtain old and new results for the regular even Orlicz--Minkowski problems; the corresponding $L_p$ version is the even $L_p$-Minkowski problem for $p>-n-1$. Moreover, employing a parabolic approximation method, we give new proofs of some of the existence results for the \textit{general} Orlicz--Minkowski problems; the $L_p$ versions are the even $L_p$-Minkowski problem for $p>0$ and the $L_p$-Minkowski problem for $p>1$. In the final section, we use a curvature flow with no global term to solve a class of $L_p$-Christoffel--Minkowski type problems.
\end{abstract}

\maketitle
\tableofcontents
\section{Introduction}
A convex body in Euclidean space is a compact convex set with non-empty interior. Write $\mathrm{K}$ for set of convex bodies, and $\mathrm{K}_o$ for the set of convex bodies containing the origin $o$. The support function of a convex body $K$ is defined by
\[h_K(u)=\sup_{v\in K}\langle u,v \rangle\quad\hbox{for}~u\in\Euc.\]
For any $u$ on the boundary of $K,$ let $\nu_K(u)$ be the set of all unit exterior normal vectors at $u.$
The surface area measure of $K$, $S_K$, is a Borel measure on the unit sphere defined by
\[S_K(\omega)=\mathcal{H}^n(\nu_K^{-1}(\omega))~\hbox{for all Borel sets }~\omega\subset\Sp.\]
Here $\mathcal{H}^n$ denotes the $n$-dimensional Hausdorff measure. When the boundary of $K$, $\partial K$, is a $C^2$ smooth, strictly convex hypersurface, then $dS_K=\sigma_nd\theta,$ where $\sigma_n$ is the product of the principal radii of curvature and $d\theta$ is the standard spherical Lebesgue measure. In this case, $h_K$, as a function on the unit sphere, is given by \[h_K(u)=\langle u,\nu_K^{-1}(u)\rangle.\]

The Minkowski problem is one of the cornerstones of the classical Brunn--Minkowski theory. It asks what are the necessary and sufficient conditions on a Borel measure $\mu$ on $\Sp$ in order for it to be the surface area measure of a convex body. The complete solution to this problem was found by Minkowski, Aleksandrov and Fenchel--Jessen (see \cite{Schneider2013a}): A Borel measure $\mu$ whose support is not contained in a closed hemisphere is the surface area measure of a convex body if and only if
\[\int_{\Sp} ud\mu(u)=o.\]
Moreover, the solution is unique up to translations.

Let $\varphi:(0,\infty)\to (0,\infty)$ be a continuous function. The \emph{general} Orlicz--Minkowski problem asks what are the necessary and sufficient conditions on a Borel measure $\mu$ on $\Sp$, such that there exists a convex body $K\in \mathrm{K}_o$ so that
\begin{align} \label{measure orlicz prob}
  \varphi(h_K)dS_K=\gamma d\mu \quad\hbox{for some constant}~\gamma>0.
\end{align}
This problem is a generalization of the $L_p$-Minkowski problem (i.e., $\varphi(s)=s^{1-p}$) which itself was put forward by Lutwak \cite{Lutwak1993a} almost a century after Minkowski's original work and stems from the $L_p$ linear combination of convex bodies. We refer the reader to \cite{Bianchi2020,Bianchi2019a,Boroczky2016b,Boroczky2012a,Boroczky2017,Chen2017a,Chou2006,Ivaki2019a,Jian2015,Lutwak1995a,Stancu2002} regarding the $L_p$-Minkowski problem and to \cite{Cianchi2009,Lutwak2002a} for applications.

A natural generalization of $L_p$ spaces are Orlicz spaces, motivating the Orlicz linear combination of convex bodies and leading to the Orlicz--Minkowski problem. We keep the discussion brief here and the reader may consult \cite{Bianchi2019b,Haberl2010,Gardner2014a,Gardner2019,Gardner2020} as well as \cite{Xi2014,Lutwak2010,Schneider2013a,Zou2014} for the origin of the Orlicz--Brunn--Minkowski theory and related concepts.

The \emph{regular} Orlicz--Minkowski problem asks what are the necessary and sufficient conditions on the smooth function $f:\Sp\to (0,\infty)$ such that there exists a convex hypersurface with support function $h$ (as a function on the unit sphere) satisfying
\begin{align} \label{orlicz prob}
 f \varphi(h)\sigma_n(h_{ij}+\delta_{ij}h)=\gamma\quad\hbox{for some constant}~\gamma>0.
\end{align}
The subscripts of $h$ denote covariant derivatives with respect to a local orthonormal frame field on $\Sp$. Our approach to solving this problem is via the flow method in the regular case \eqref{orlicz prob} and by parabolic approximation in the general case \eqref{measure orlicz prob}.
\section{Results}
\subsection{Curvature flows}
Let \[x_0:M^n\to \Euc\] be a smooth parametrization of a closed, strictly convex hypersurface $M_0$ with the origin of $\Euc$ in its interior. One of the flows that we are interested in is the family of hypersurfaces $\{M_t\}$ given by the smooth map
\[x:M^n\times [0,T)\to \Euc\]
satisfying the initial value problem
\begin{align}\label{F-param}
\left\{
  \begin{array}{ll}
  \partial_tx= f (\nu)\langle x,\nu\rangle\frac{\varphi(\langle x,\nu\rangle)}{\mathcal{K}}\nu-\zeta(t) x; \\
   x(\cdot,0)=x_0(\cdot),
  \end{array}
\right.
\end{align}
where $\mathcal{K}(\cdot,t)$ and $\nu(\cdot,t)$ are respectively the Gauss curvature and the outer unit normal vector of $M_t:=x(M^n,t)$, $ f:\Sp\to (0,\infty) $ and $\varphi:(0,\infty)\to (0,\infty)$ are smooth functions and
\[\zeta(t):=\frac{\int\langle x,\nu\rangle d\mu_{M_{t}}}{\int \frac{\langle x,\nu\rangle \mathcal{K}}{ f (\nu)\varphi(\langle x,\nu\rangle)} d\mu_{M_t}}.\]
To explain our interest in studying this flow, we need some definitions and notation.

By a direct calculation, along the flow \eqref{F-param} the support functions \[h:\Sp\times [0,T)\to \mathbb{R}\] of the $M_t$ (as long as they are strictly convex) satisfy
\begin{align}\label{nf}
\left\{
  \begin{array}{ll}
\partial_th= f  h\varphi(h)\sigma_n-\zeta h;\\
h(u,0)=h_{M_0}.
  \end{array}
\right.
\end{align}
Stationary solutions of (\ref{nf}), whenever they exist, are exactly the solutions of the regular Orlicz--Minkowski problem and thus the limits of the flow hypersurfaces are solutions of (\ref{orlicz prob}).

For some choices of $f$ and $\varphi(s)=s^{1-p}$, the flow (\ref{F-param}) becomes homogeneous and was considered in \cite{Andrews1998,Bryan2018d,Chow1997b,Gerhardt2014,Ivaki2019,Ivaki2016b,Ivaki2015c,Ivaki2014,Ivaki2013a,Li2010,Scheuer:07/2016,Schnuerer:/2006,Stancu2012}. However, when it comes to non-homogeneous flows, the literature on geometric flows is not very rich and there are few works in this direction; e.g., \cite{Bertini2018,Chow1998,Chow1997b,Chou2000,Kroner2017a,Li2019a,Liu2020}. Since the Orlicz-Minkowski problem admits solutions in the non-homogeneous case, it is desirable to remove the homogeneity assumption in the flow. The conditions we place on \(f,\varphi\) are similar to those used previously in the literature on the Orlicz--Minkowski problem mentioned in the introduction while some are new (\Cref{main remark1} and \Cref{main remark2} below).
\begin{theorem}\label{main theorem1}
Suppose
\begin{enumerate}
  \item $f:\Sp \to (0,\infty)$ is smooth and even, $f(u)=f(-u)\quad \forall u\in\Sp$,
  \item $\varphi:(0,\infty)\to (0,\infty)$ is smooth,
  \item either \begin{description}
          \item[a] $\phi(s)=\int_0^s\frac{1}{\varphi(t)}dt$ exists for all $s>0,$ and $\lim\limits_{s\to\infty}\phi(s)=\infty,$
          \item[b] or $\phi(s)=\int_1^s\frac{1}{\varphi(t)}dt$ satisfies $\lim\limits_{s\to\infty}\phi(s)=\infty,$
          and for some $\hat{C}_1>0,$ we have
           \[\int_{\Sp}\frac{\phi(|\langle u,\theta\rangle|)}{f}d\theta\geq -\hat{C}_1\quad \forall u\in\Sp,\]
           \item[c] or $\phi(s)=\int_s^{\infty}\frac{1}{\varphi(t)}dt$ exists for all $s>0,$ 
           and for some $q\in (-n-1,0)$, we have
           \[\limsup_{s\to0^+}\frac{\phi(s)}{s^q}<\infty,\]
          \item[d] or $\phi(s)=\int_s^{\infty}\frac{1}{\varphi(t)}dt$ exists for all $s>0,$ $\lim\limits_{s\to0^+}\phi(s)=\infty,$
and for some $\hat{C}_2>0,$ we have
           \[\int_{\Sp}\frac{\phi(|\langle u,\theta\rangle|)}{f}d\theta\leq \hat{C}_2\quad \forall u\in\Sp.\]
        \end{description}
\end{enumerate}
In the cases, (3-a), (3-b), and (3-c), let $M_0$ be $o$-symmetric. In the case (3-d), we choose an $o$-symmetric $M_0$ such that
\[\int_{\Sp}\frac{\phi(h_{M_0})}{f}d\theta>\hat{C}_2.\]
Then there exists a smooth, strictly convex solution $M_t$ to (\ref{F-param}) and it subconverges in $C^{\infty}$ to an $o$-symmetric, smooth, strictly convex solution of the regular even Orlicz--Minkowski problem.
\end{theorem}
\begin{remark}\label{main remark1}~
\begin{enumerate}
  \item When $\varphi(s)=s^{1-\ell},$ assumptions (3-a), (3-b), (3-c), (3-d) are satisfied, in order, for $\ell>0$, $\ell=0$, $\ell\in (-n-1,0)$, $-1<\ell<0$. To see this, note that
\begin{align*}
&\int_{\Sp} \log|x_i|d\theta>-\infty,\\
&\int_{\Sp}\frac{1}{|x_i|^\ell}d\theta<\infty\quad\mbox{for}\quad 0<\ell<1,
\end{align*}
where $(x_1,\ldots,x_{n+1})$ denotes the Euclidean coordinates.
Hence, Theorem \ref{main theorem1} generalizes \cite[Thm. 1]{Bryan2018d} to the Orlicz setting.
\item Since $\phi$ as defined in (3-a) is increasing, we have
\[\int_{\Sp}\frac{\phi(|\langle u,\theta\rangle|)}{f}d\theta <\infty\quad \forall u\in \Sp.\]
Therefore, the integral conditions in (3-b) and (3-d) compared to the condition (3-a) are not more restrictive.
  \item The cases (3-b) and (3-d) yield new existence results for the regular even Orlicz--Minkowski problem, while the existence results corresponding to (3-a) and (3-c) can be deduced from \cite{Bianchi2019b,Haberl2010}.
\item Due to the presence of $\varphi$ (possibly non-homogenous) and $\zeta$ in the speed of the flow (\ref{F-param}), we could not employ the method of \cite{Andrews1997a} to promote subconvergence to full convergence.
\end{enumerate}
\end{remark}
We slightly modify the flow (\ref{F-param}) to treat a class of regular Orlicz--Minkowski problems without the assumption that $f$ is even, i.e., we no longer assume \[f(u)=f(-u)\quad \forall u\in\Sp.\]
In order to serve this purpose, we consider the flow
\begin{align}\label{F-param2}
\left\{
  \begin{array}{ll}
  \partial_tx= f (\nu)\langle x,\nu\rangle\frac{\varphi(\langle x,\nu\rangle)}{\mathcal{K}}\nu- x; \\
   x(\cdot,0)=x_0(\cdot),
  \end{array}
\right.
\end{align}
where again $x_0:M^n\to \Euc$ is a smooth parametrization of a closed, strictly convex hypersurface $M_0$ with the origin of $\Euc$ in its interior.
Here compared to (\ref{F-param}), we dropped the $\zeta$-factor. In the appendix, we will use another curvature flow without a global term  (\ref{F-param4}) to treat a class of $L_p$-Minkowski-Christoffel type problems. See also \cite{BIS2020,Gerhardt1996,Gerhardt2003a,Gerhardt2006b,Gerhardt2008} for flows without global terms.
\begin{theorem}\label{main theorem2}
Suppose \begin{enumerate}
  \item $\varphi:(0,\infty)\to (0,\infty)$ is smooth,
  \item $\limsup\limits_{s\to\infty}s^n\varphi(s)<\frac{1}{f}<\liminf\limits_{s\to 0^+} s^n\varphi(s).$
\end{enumerate}
Then there exists a smooth, strictly convex solution $M_t$ of (\ref{F-param2}) and it converges in $C^{\infty}$ to a smooth, strictly convex solution of (\ref{orlicz prob}) with positive support function and constant $\gamma=1.$
\end{theorem}
\begin{remark}\label{main remark2}~
\begin{enumerate}
  \item In the case special case $\varphi(s)=s^{1-p}$, Theorem \ref{main theorem2} finds the solutions of the regular $L_p$-Minkowski problems for $p>n+1$.
  \item In \cite[p. 42]{Chou2000}, using a logarithmic curvature flow, an existence result was obtained under the assumption (2) and that $\varphi$ is non-increasing. See also \cite{Liu2020}, where a similar result was recently obtained.
\end{enumerate}
\end{remark}
\subsection{General measures}
\begin{definition}
A Borel measure $\mu$ on $\Sp$ is said to be even if it assumes the same values on antipodal Borel sets. We say $\mu$ is invariant under a subgroup $G$ of the orthogonal group $O(n+1)$ if $\mu(g\omega)=\mu(\omega)$ for all Borel sets $\omega\subseteq\Sp.$
\end{definition}
\begin{theorem}\label{main theoremHLYZ}
Suppose
\begin{enumerate}
  \item $\varphi:(0,\infty)\to (0,\infty)$ is a continuous function,
  \item $\phi(s)=\int_0^s\frac{1}{\varphi(t)}dt$ exists for all $s>0$ and $\lim\limits_{s\to\infty}\phi(s)=\infty.$
\end{enumerate}
Let $\mu$ be a finite even Borel measure on $\Sp$ whose support is not contained on a great subsphere.
Then there exists an $o$-symmetric convex body such that
\begin{align*}
\varphi(h_K)dS_K=\gamma d\mu \quad\hbox{for some constant}~\gamma>0.
\end{align*}
\end{theorem}
This theorem was first proved in \cite[Thm.~2]{Haberl2010} and contains the general even $L_p$-Minkowski problem for $p>0$. The method of HLYZ is a variational argument that finds a minimizing body of a suitable functional in a \textit{certain} class of origin-symmetric ($o$-symmetric) convex bodies. We treated the regular version of this theorem by using the curvature flow in Theorem \ref{main theorem1}. The general case will be treated by a simple parabolic approximation, cf., Section \ref{subsec approx}.

In the next theorem, the assumption that the measure $\mu$ is even is dropped. To prove this theorem, we use the flow (\ref{F-param}) and a version of Chou--Wang's approximation argument \cite{Chou2000} adapted to the parabolic setting \cite{Haodi2019} for treating the general $L_p$-Minkowski problem.
\begin{theorem}\label{main theoremGeneral}
Suppose
\begin{enumerate}
  \item $\varphi:(0,\infty)\to (0,\infty)$ is continuous and $\lim\limits_{s\to0^+}\varphi(s)=\infty,$
  \item $\phi(s)=\int_0^s\frac{1}{\varphi(t)}dt$ exists for all $s>0$ and $\lim\limits_{s\to\infty}\phi(s)=\infty.$
\end{enumerate}
Let $\mu$ be a finite Borel measure on $\Sp$ whose support is not contained in a closed hemisphere.
Then there exist $K\in \mathrm{K}_o$ and a constant $\gamma>0$ such that
\begin{align*}
 dS_K= \frac{\gamma}{\varphi(h_K)}d\mu.
\end{align*}
Moreover, if $\mu$ is invariant under a closed group $G\subset O(n+1),$ then $K$ can be chosen to be invariant under $G.$
\end{theorem}
This theorem first appeared in \cite{Huang2012} and for the special case $\varphi(s)=s^{1-p}$, it includes the $L_p$-Minkowski problem for $p>1$. In \cite{Huang2012}, the problem is first solved for discrete measures and then by approximating the general measure by discrete measures; cf., Section \ref{subsec approx} for the sketch and the full details. We solve it first in the regular case and then by approximating the general measures by the regular case. See also Wu--Xi--Leng \cite{Wu2019,Wu2018} for the discrete case of the Orlicz--Minkowski problem, and \cite{Jian2019} for an existence result of solutions to general Orlicz--Minkowski problem which contains the $L_p$-Minkowski problem for $0<p<1.$ Note that the role of the constant $\gamma$ is essential in general; see \cite{Yijing2015} for a non-existence result when $\gamma$ is dropped.
\section{Flow hypersurfaces}\label{reg est}
\subsection{Regularity estimates}
In the sequel, $\bar{g}$ and $\n$ denote respectively the standard round metric and the Levi-Civita connection of $\Sp$. The principal radii of curvature are the eigenvalues of the matrix
\begin{equation}\label{rij}
r_{ij}:=\n_i\n_j h+\bar{g}_{ij}h
\end{equation}
with respect to $\bar{g}$. For convenience, we put
\begin{align*}\eta(t)&=\left\{
          \begin{array}{ll}
            \zeta(t) & \hbox{flow}~(\ref{F-param}),\\
            1 & \hbox{flow}~(\ref{F-param2}),
          \end{array}
        \right.\\
 \Theta&= f  h\varphi(h),\quad \mathcal{L}=\partial_t-\Theta \sigma_n^{ij}\n_i\n_j, \quad \rho=\sqrt{h^2+|\n h|^2}.
\end{align*}
\begin{lemma}\label{ev eqs}
The following evolution equations hold along the flows.
\begin{align*}
\mathcal{L}h=&(1-n)\Theta \sigma_n+\Theta h\sigma_n^{ij}\bar{g}_{ij}-\eta h,\\
\mathcal{L}\frac{\rho^2}{2}=& (n+1)h\Theta \sigma_n+\sigma_n\bar{g}^{ij}\n_ih\n_j\Theta-\Theta\sigma_n^{ij}r_i^kr_{jk}-\eta \rho^2,\\
\mathcal{L}(\Theta\sigma_n)
=&\left(\frac{1}{h}+\frac{\varphi'(h)}{\varphi(h)}\right)(\Theta\sigma_n)^2 +\Theta^2\sigma_n\sigma_n^{ij}\bar{g}_{ij}\\
&-\eta\left(n+1+\frac{h\varphi'(h)}{\varphi(h)}\right)\Theta\sigma_n.
\end{align*}
\end{lemma}
\begin{proof}
We have
$
\partial_{t}h=\Theta\sigma_{n}-\eta h
$
and hence the first equation follows from the $n$-homogeneity and \eqref{rij}.
For the second equation, in an orthonormal frame that diagonalizes $r_{ij},$
\begin{align*}
\n_a\n_b\frac{\rho^2}{2}=&h\n_a\n_bh+\n_ah\n_bh\\
&+\n^mh\n_a\n_b\n_mh+\n_a\n^mh\n_b\n_mh\\
=&h(r_{ab}-\bar{g}_{ab}h)+\n_ah\n_bh\\
&+\n^mh\n_a(r_{bm}-\bar{g}_{bm}h)\\
&+(r_{a}^m-\delta_{a}^mh)(r_{bm}-\bar{g}_{bm}h).
\end{align*}
Tracing with respect to $\sigma_n^{ab}$ gives
\begin{align*}
\sigma_n^{ab}\n_a\n_b\frac{\rho^2}{2}=&h(n\sigma_n-\sigma_n^{ab}\bar{g}_{ab}h)+\sigma_n^{ab}\n_ah\n_bh\\
&+\n^mh\n_m\sigma_n-\sigma_n^{ab}\n_ah\n_bh\\
&+\sigma_n^{ab}r_{a}^mr_{bm}+\sigma_n^{ab}\bar{g}_{ab}h^2-2nh\sigma_n\\
=&-nh\sigma_n+\n^mh\n_m\sigma_n+\sigma_n^{ab}r_{a}^mr_{bm}.
\end{align*}
Now the second evolution equation follows from
\begin{align*}
\partial_t\frac{\rho^2}{2}=&h\Theta\sigma_n-\eta\rho^2+\bar{g}^{ij}\sigma_n\n_ih\n_j\Theta+\bar{g}^{ij}\Theta\n_ih\n_j\sigma_n.
\end{align*}
Finally there holds
\begin{align*}\begin{split}
\partial_{t}(\Theta\sigma_{n})=&f\varphi\sigma_{n}\partial_{t}h+fh\varphi' \sigma_{n}\partial_{t}h+\Theta \sigma_{n}^{ij}(\n_{i}\n_{j}\partial_{t}h+\bar g_{ij}\partial_{t}h)\\
					=&f\varphi\sigma_{n}(\Theta\sigma_{n}-\eta h)+fh\varphi' \sigma_{n}(\Theta\sigma_{n}-\eta h)\\
					&+\Theta \sigma_{n}^{ij}(\n_{i}\n_{j}(\Theta\sigma_{n}-\eta h)+\bar g_{ij}(\Theta\sigma_{n}-\eta h))\\
					=&\Theta\sigma_{n}^{ij}\n_{i}\n_{j}(\Theta\sigma_{n})+f\varphi\sigma_{n}(\Theta\sigma_{n}-\eta h)\\
					&+fh\varphi'\sigma_{n} (\Theta\sigma_{n}-\eta h)+\Theta^2\sigma_{n}^{ij}\bar g_{ij} \sigma_{n}\\
					&-\eta\Theta\sigma_{n}^{ij}(\n_{i}\n_{j}h+\bar{g}_{ij}h).
\end{split}
\end{align*}
Hence,
\begin{align}\begin{split}
\mathcal{L}(\Theta\sigma_{n})=&\frac{1}{h}\Theta^{2}\sigma_{n}^{2}
+\frac{\varphi'}{\varphi}\Theta^{2}\sigma_{n}^{2}+\sigma_{n}^{ij}\bar g_{ij}\Theta^{2}\sigma_{n}\\
					&-n\eta\Theta\sigma_{n}-\eta \Theta\sigma_{n}-\eta\frac{\varphi'}{\varphi}\Theta h\sigma_{n}\\
					=&\left(\frac{1}{h}+\frac{\varphi'}{\varphi}\right)\Theta^{2}\sigma_{n}^{2}+\sigma_{n}^{ij}\bar g_{ij}\Theta^{2}\sigma_{n}\\
					&-\left(1+n+\frac{\varphi'}{\varphi}h\right)\eta \Theta\sigma_{n}.
\end{split}\end{align}
\end{proof}
Define
\begin{align*}
\mathcal{E}[h]=\int_{\Sp} \frac{\phi(h)}{ f }d\theta,\quad\mathcal{E}(t)=\mathcal{E}[h_{M_t}].
\end{align*}
\begin{lemma}\label{monotonicity}
Along the flow (\ref{F-param}) we have
\begin{align*}
\frac{d}{dt}V\geq 0,\quad \frac{d}{dt}\mathcal{E}=0.
\end{align*}
Along the flow (\ref{F-param2}) we have
\begin{align*}
\frac{d}{dt}(V-\mathcal{E})\geq 0.
\end{align*}
Here $V$ denotes the enclosed volume of $M_t$. In both cases, the monotonicity is strict unless the solution is stationary.
\end{lemma}
\begin{proof}
For the flow (\ref{F-param}), recalling that
\[\zeta(t):=\frac{\int_{\mathbb{S}^n}h\sigma_nd\theta}{\int_{\mathbb{S}^n}\frac{h}{f\varphi(h)}d\theta},\]
we obtain
\begin{align}\label{orlicz width}
\frac{d}{d t}\mathcal{E}(t)=&\pm\int_{\Sp}\left(f  h\varphi(h)\sigma_n-\frac{\int_{\Sp} h \sigma_nd\theta}{\int_{\Sp}\frac{h}{ f  \varphi(h)}d\theta}h\right)\frac{1}{f \varphi(h)}d\theta=0.
\end{align}
Here $+$ is for the cases (3-a) and (3-b), and $-$ is for the cases (3-c) and (3-d).
Using \[V = \frac{1}{n+1} \int_{\Sp} h \sigma_n d\theta,\] the divergence theorem, the $n$-homogeneity of $\sigma_n$, and that
$\bar{\nabla}_i\sigma_n^{ij}=0$
(cf., \cite[Lem. 2-12]{Andrews1994}), we obtain
\begin{align}\label{lower bound on volume}
\frac{1}{n+1}\frac{d}{d t}\int_{\Sp}h\sigma_nd\theta=&\int_{\Sp}\sigma_n\partial_thd\theta\nonumber\\
=&\int_{\Sp}\left(\sigma_n-\zeta(t)\frac{1}{f\varphi(h)}\right)\partial_thd\theta\nonumber\\
=&\int_{\Sp}\frac{(\partial_th)^2}{fh\varphi(h)}d\theta\geq0.
\end{align}
Equality holds precisely when $h$ solves (\ref{orlicz prob}).

For the flow (\ref{F-param2}), note that
\begin{align}
\frac{d}{dt}(V-\mathcal{E})&=\int_{\Sp}\left(\sigma_n-\frac{1}{f\varphi(h)}\right)\partial_thd\theta \nonumber\\
&=\int_{\Sp}\frac{(\partial_th)^2}{fh\varphi(h)}d\theta\geq 0. \label{monotonicity-flow2}
\end{align}
The equality holds precisely when $h$ solves (\ref{orlicz prob}) with $\gamma=1.$
\end{proof}
We write $w_-$ and $w_+$ respectively for the minimum width and the maximum width of a closed, convex hypersurface (or a convex body) with support function $h$. They are defined as
\[w_+=\max_{u\in\Sp}(h(u)+h(-u)),\quad w_-=\min_{u\in\Sp}(h(u)+h(-u)).\]
\begin{lemma}\label{C0 estimate}Under either flow \eqref{F-param} and \eqref{F-param2} and the corresponding assumptions, there are constants $a,b$ such that $a\leq h(\cdot,t)\leq b.$ Moreover, $\eta(t)$ is uniformly bounded above, and below away from zero.
\end{lemma}
\begin{proof}
Suppose $R:=\max h_{M_t}$ is attained at the north pole $e_{n+1}$.

Theorem \ref{main theorem1}, case (3-a):
Due to convexity,
\[h_{M_t}(x)\geq Rx_{n+1}\quad ~\mbox{on}~\{x_{n+1}>0\}\cap \Sp.\]
Since $\phi\geq 0$ is non-decreasing and $\mathcal{E}(t)$ constant in time, we obtain
\begin{align}
\phi(\tfrac{1}{2}R)\int_{\{x_{n+1}>\frac{1}{2}\}\cap \Sp}\frac{1}{f}d\theta\leq \mathcal{E}(t)= \mathcal{E}(0).
\end{align}
Due to $\lim_{s\to\infty}\phi(s)=\infty$, we see that $h(\cdot,t)$ is uniformly bounded above. Let $V(M_t)$ denote the volume of the enclosed region by $M_t.$ To prove the uniform lower bound of $h$, note that 
\[V(M_t)\leq w_{-}(M_t)w_+(M_t)^n.\]
Since $M_t$ is $o$-symmetric,
$w_+=2\max h$ and $w_-=2\min h.$
The lower bound of $h$ now follows since $V(M_t)$ is non-decreasing along the flow.
The lower and upper bounds on $h$ imply bounds on $\eta.$

Theorem \ref{main theorem1}, case (3-b): Suppose $R>2$. Due to convexity, \[h_{M_t}\geq R|x_{n+1}|.\] Moreover, $\phi$ is non-decreasing,
\[\left\{
    \begin{array}{ll}
      \phi(s)\geq 0 & s\geq 1 \\
      \phi(s)\leq 0 & 0<s\leq 1,
    \end{array}
  \right.
\]
and $\mathcal{E}(t)$ constant in time. Hence,
\begin{align*}
\mathcal{E}(0)\geq& \int_{\Sp}\frac{\phi(R|x_{n+1}|)}{f}d\theta\\
\geq& \phi(\tfrac{1}{2}R)\int_{\{|x_{n+1}|>\frac{1}{2}\}\cap \Sp}\frac{1}{f}d\theta+
\int_{\{|x_{n+1}|\leq\frac{1}{2}\}\cap \Sp}\frac{\phi(R|x_{n+1}|)}{f}d\theta\\
\geq& \phi(\tfrac{1}{2}R)\int_{\{|x_{n+1}|>\frac{1}{2}\}\cap \Sp}\frac{1}{f}d\theta+
\int_{\{|x_{n+1}|\leq\frac{1}{2}\}\cap \Sp}\frac{\phi(|x_{n+1}|)}{f}d\theta\\
\geq& \phi(\tfrac{1}{2}R)\int_{\{|x_{n+1}|>\frac{1}{2}\}\cap \Sp}\frac{1}{f}d\theta+
\int_{\Sp}\frac{\phi(|x_{n+1}|)}{f}d\theta\\
\geq& \phi(\tfrac{1}{2}R)\int_{\{|x_{n+1}|>\frac{1}{2}\}\cap \Sp}\frac{1}{f}d\theta-
\hat{C}_1.
\end{align*}
Since $\lim_{s\to\infty}\phi(s)=\infty$, $R$ remains uniformly bounded above. This in turn implies the lower and upper bounds on support functions and $\eta.$

Theorem \ref{main theorem1}, case (3-c): Since $M_t$ is $o$-symmetric body and the volume is non-decreasing, the lower and upper bound on the support functions and $\eta$ follow from the proof of \cite[Lem. 4.3]{Bianchi2019b}.

Theorem \ref{main theorem1}, case (3-d): Suppose $R>1$. Since $\phi\geq 0$ is non-increasing
and $\mathcal{E}(t)=\mathcal{E}(0)$,
\begin{align*}
\mathcal{E}(0)\leq& \int_{\Sp}\frac{\phi(R|x_{n+1}|)}{f}d\theta\\
\leq& \phi(\tfrac{1}{2}R)\int_{\{|x_{n+1}|>\frac{1}{2}\}\cap \Sp}\frac{1}{f}d\theta+
\int_{\{|x_{n+1}|\leq\frac{1}{2}\}\cap \Sp}\frac{\phi(R|x_{n+1}|)}{f}d\theta\\
\leq& \phi(\tfrac{1}{2}R)\int_{\{|x_{n+1}|>\frac{1}{2}\}\cap \Sp}\frac{1}{f}d\theta+
\int_{\{|x_{n+1}|\leq\frac{1}{2}\}\cap \Sp}\frac{\phi(|x_{n+1}|)}{f}d\theta\\
\leq& \phi(\tfrac{1}{2}R)\int_{\Sp}\frac{1}{f}d\theta+
\int_{\Sp}\frac{\phi(|x_{n+1}|)}{f}d\theta.
\end{align*}
Note that as $R\to\infty$,  due to $\lim_{s\to\infty}\phi(s)=0,$ we obtain
\[\mathcal{E}(0)\leq \hat{C}_2.\]
By our choice of $M_0$, this last inequality is violated.

Theorem \ref{main theorem2}: To get a uniform upper bound on $h$, note that at a maximum of $h$ we have
\[\partial_th_{\max}\leq h_{\max}(f_{\max}\varphi(h_{\max})h_{\max}^{n}-1).\] The right-hand side will be negative if $h_{\max}\to\infty.$
To get a lower bound on $h$ away from zero, we can argue similarly.
\end{proof}
In the following two lemmas, we will use two auxiliary functions from \cite{Kroner2019a,Ivaki2019} to find uniform lower and upper bounds on the Gauss curvature.
\begin{lemma}
We have $\sigma_n\geq c$ for some positive constant.
\end{lemma}
\begin{proof}
The evolution equation of $\chi:=\log(\Theta \sigma_n)-A\frac{\rho^2}{2}$ is given by
\begin{align*}
\mathcal{L}\chi=&\Theta \sigma_n^{ij}\n_i\log(\Theta\sigma_n)\n_j\log(\Theta\sigma_n)\\
&+\left(\frac{1}{h}+\frac{\varphi'(h)}{\varphi(h)}\right)\Theta \sigma_n+\Theta\sigma_n^{ij}\bar{g}_{ij}\\
&-\eta\left(n+1+\frac{h\varphi'(h)}{\varphi(h)}\right)\\
&-(n+1)A h\Theta \sigma_n-A\sigma_n\bar{g}^{ij}\n_ih\n_j\Theta\\
&+A\Theta\sigma_n^{ij}r_i^kr_{jk}+A\eta \rho^2.
\end{align*}
Dropping some positive terms gives
\begin{align*}
\mathcal{L}\chi\geq&\left(\frac{1}{h}+\frac{\varphi'(h)}{\varphi(h)}\right)\Theta \sigma_n-\eta\left(n+1+\frac{h\varphi'(h)}{\varphi(h)}\right)\\
&+\frac{A}{2}\eta \rho^2-(n+1)A h\Theta \sigma_n-A\sigma_n\bar{g}^{ij}\n_ih\n_j\Theta.
\end{align*}
By the $C_0$-estimate, if $A$ is chosen large enough, the right-hand side of the above equation is strictly positive provided $\min \sigma_n\to 0$. Thus $\sigma_n$ is uniformly bounded below away from zero.
\end{proof}
\begin{lemma}
There is a constant $d$ such that $\sigma_n\leq d.$
\end{lemma}
\begin{proof}
For $\ep>0$ sufficiently small, consider the auxiliary function
\[\chi:=\frac{f\varphi\sigma_n}{1-\ep \frac{\rho^2}{2}}.\]
We have
\begin{align}\label{soliton eq}
\mathcal{L}\frac{\Theta \sigma_n}{h}=&2\Theta\sigma_n^{ij}\n_i\log h\n_j\left(\frac{\Theta \sigma_n}{h}\right)
\nonumber
\\&+\left(n+\frac{h\varphi'(h)}{\varphi(h)}\right)\left(\frac{\Theta \sigma_n}{h}\right)\left(\frac{\Theta \sigma_n}{h}-\eta\right).
\end{align}
Therefore, by Lemma \ref{ev eqs} and (\ref{soliton eq}), at any point with $t>0$ where $\chi$ attains a maximum,
\begin{align*}
\partial_t\chi\leq&\frac{1}{1-\ep \frac{\rho^2}{2}}\Biggl(2\Theta\sigma_n^{ij}\n_i\log h\n_j\left(\frac{\Theta \sigma_n}{h}\right)
\\&+\left(n+\frac{h\varphi'(h)}{\varphi(h)}\right)\left(\frac{\Theta \sigma_n}{h}\right)\left(\frac{\Theta \sigma_n}{h}-\eta\right)\Biggr)\\
&+\frac{\ep f\varphi(h)\sigma_n}{(1-\ep \frac{\rho^2}{2})^2}\left((n+1)h\Theta \sigma_n+\sigma_n\bar{g}^{ij}\n_ih\n_j\Theta-\Theta\sigma_n^{ij}r_i^kr_{jk}-\eta \rho^2\right).
\end{align*}
Using $\n \chi=0$ we have \[\n_j\log\left(\tfrac{\Theta \sigma_n}{h}\right) = \n_j \log\left(\tfrac{1}{1-\ep^2\rho^2/2}\right).\] Substituting in
\[
\n_j \frac{\rho^2}{2} = h\n_j h + \bar{g}^{kl}\n_j \n_k h \n_l h = \bar{g}^{kl} r_{jl} \n_k h.
\]
and from the convexity of the hypersurface, we then see that the term $\sigma_n^{ij}\n_i\log h\n_j\left(\frac{\Theta \sigma_n}{h}\right)$ is non-positive. Thus due to AM-GM inequality,
\[\sigma_n^{ij}r^k_i r_{jk}\geq c\sigma_n^{1+\tfrac{1}{n}},\]
and the $C_0$-estimate, we obtain
\[\partial_{t} \chi\leq c_1\chi+c_2\chi^2-c_3\chi^{2+\frac{1}{n}}\]
for some positive constants $c_1,c_2,c_3$. Here we identified $\chi$ with $\max \chi$.
 From this the claim follows.
\end{proof}
\begin{lemma}
The principal curvatures are uniformly bounded above, and below away from zero.
\end{lemma}
\begin{proof}
To obtain the lower and upper curvature bounds, we can argue exactly as in the proof of \cite[Lem.~8]{Bryan2018d} where we applied the maximum principle to the evolution equation of $\chi:=\log( \|W\|/ h^{\alpha})$. Here $\|W\|^2$ is the sum of squares of the principal curvatures and $\alpha>0$ is a suitable constant. A different argument is also given in the proof of \cite[Lem.~8.2]{Haodi2019}.
\end{proof}
\subsection{Convergence}\label{sec:conv}
In the previous section, the $C^2$-estimates were obtained for either flow under their corresponding assumptions. Now the higher order regularity estimates follow from the theory of parabolic differential equations; see for details \cite{McCoy2005,Makowski2013}. Therefore, the maximal time interval is unbounded.

Due to Lemma \ref{monotonicity}, subconvergence for the flows (\ref{F-param}) and (\ref{F-param2}) to stationary solutions is standard: by the upper bound on the support functions, there is a constant $C$, depending only on $M_0,$ such that
\[\int_{\mathbb{S}^n} h\sigma_nd\theta\leq C.\]
In view of (\ref{lower bound on volume}) and (\ref{monotonicity-flow2}), this implies that
\[\liminf_{t\to\infty} \int_{\Sp}(\partial_t h)^2=0.\]
Hence, the flow hypersurfaces subconverge to a stationary solution.

To obtain full convergence for the flow (\ref{F-param2}), let us put
\[\mathcal{F}=V-\int\frac{\phi(h)}{f}d\theta.\]
Recall from (\ref{monotonicity-flow2}) that
\begin{align*}
\frac{d}{dt}\mathcal{F}=\int_{\Sp}\frac{(\partial_th)^2}{fh\varphi(h)}d\theta.
\end{align*}
Note that
\[\nabla^{L^2(\Sp)} \mathcal{F}=\sigma_n-\frac{1}{f\varphi(h)}=\frac{\partial_th}{fh\varphi(h)}.\]
Hence due to our $C_0$-estimates,
\begin{align*}
\frac{d}{dt}\mathcal{F}\geq c_0\|\nabla^{L^2(\Sp)}\mathcal{F}\|_{L^2(\Sp)}\|\partial_th\|_{L^2(\Sp)}.
\end{align*}
Now we can argue as in \cite{Andrews1997a,Guan2003c} to promote subconvergence to full convergence.
\section{General measures}\label{general measure}
\subsection{Notions and notation}
Let $\phi: [0,\infty)\to[0,\infty)$ be an increasing function, continuously differentiable on $(0,\infty)$ with positive derivative, and satisfying
$\lim_{s\to\infty}\phi(s)=\infty.$ For a finite non-zero Borel measure $\mu$ and a continuous, nonnegative function $f$ on $\Sp$, the Orlicz norm $\|f\|_{\phi,\mu}$ is defined by
\begin{align*}
\|f\|_{\phi,\mu}=\inf\left\{\lambda>0: \frac{1}{|\mu|}\int_{\Sp}\phi\left(\frac{f}{\lambda}\right)d\mu\leq \phi(1)\right\}.
\end{align*}
Here $|\mu|=\mu(\Sp).$

Note that in general the Orlicz norm does not satisfy a triangle inequality and the case $\phi(t)=t^p$ gives the normalized $L^p$ norm.
The Orlicz norm satisfies the following properties:
\begin{align}\label{basic prop orlicz}
\|cf\|_{\phi,\mu}&=c\|f\|_{\phi,\mu}\quad \forall c\in (0,\infty),\\
f\leq g&\Rightarrow \|f\|_{\phi,\mu}\leq \|g\|_{\phi,\mu}.\nonumber
\end{align}

\begin{definition}~
\begin{enumerate}
  \item The Hausdorff distance of two convex sets $K,L$ is defined by
\[d_{\mathcal{H}}(K,L)=\max_{u\in\Sp}|h_K(u)-h_L(u)|.\]
  \item A sequence $\{K_i\}_{i\in \mathbb{N}}\subset \mathrm{K}$ converges to a convex body $K$ if \[\lim_{i\to\infty}d_{\mathcal{H}}(K_i,K)=0.\]
  \item For $v\in\Sp$, define \[h_{\hat{v}}(u)=\langle u,v\rangle_+=\max\{0,\langle u,v\rangle\}\quad \forall u\in \Sp.\]
  \item Given a function $f:\Sp\to (0,\infty)$, we define the measure \[d\mu_f=\frac{1}{f}d\theta.\]
\end{enumerate}
\end{definition}

\subsection{Parabolic Approximation}\label{subsec approx}
We begin this section by sketching the proof of Theorem \ref{main theoremGeneral}. Assume $\varphi$ is smooth.

Step 1: We perturb $\varphi$ to $\varphi_{\varepsilon}$ such that \[\varphi_{\varepsilon}(s)=s^{-n-\varepsilon}, \quad \forall s\in (0,\varepsilon].\]

Step 2: Suppose $0<f\in C^{\infty}(\mathbb{S}^n).$ Using a suitable curvature flow, we find a smooth, strictly convex hypersurface $M_{\varepsilon}$ with positive support function such that
\[f\varphi_{\varepsilon}(h)\sigma_n=\frac{\int_{\Sp} h\sigma_nd\theta}{\int_{\Sp}\frac{h}{\varphi_{\ep}(h)}d\mu_f}.\]

Step 3: Take $\varepsilon=1/i$ and set $\varphi_i=\varphi_{\varepsilon}.$ Applying Step 2 to $f$ and $\varphi_i,$ we find $K_i$ (with the origin in its interior) such that
\begin{align*}
 dS_{K_i}= \frac{\gamma_i}{\varphi_i(h_{K_i})}d\mu_{f}\quad \mbox{for a constant}~\gamma_i>0.
\end{align*}

Moreover, we show that the minimum and maximum width of $K_i$ as well as $\gamma_i$ are uniformly bounded above and below away from zero, such that these bounds, for $i$ sufficiently large, depend only on $\varphi$ and $\mu_f$. Thus, letting $i\to\infty,$ we can find a convex body $K\in \mathrm{K}_o$ such that
\begin{align*}
 dS_K= \frac{\gamma}{\varphi(h_K)}d\mu_f\quad \mbox{for a constant}~\gamma>0.
\end{align*}
Step 4: Choose $0<f_i\in C^{\infty}(\mathbb{S}^n)$ such that $\mu_{f_i}\to \mu$ weakly as $i\to \infty$. By the conclusion of Step 3, we find $K_i\in \mathrm{K}_o$ such that
\begin{align*}
 dS_{K_i}= \frac{\gamma_i}{\varphi(h_{K_i})}d\mu_{f_i}\quad \mbox{for a constant}~\gamma_i>0.
\end{align*}

Moreover, $w_{\pm}(K_i)$ and $\gamma_i$ are uniformly bounded above and below away from zero, and these bounds, for $i$ sufficiently large, depend only on $\varphi$ and $\mu$. Thus, letting $i\to\infty,$ we find $K\in \mathrm{K}_o$ solving
\begin{align*}
 dS_K= \frac{\gamma}{\varphi(h_K)}d\mu\quad \mbox{for a constant}~\gamma>0.
\end{align*}
A further approximation allows us to assume $\varphi$ is continuous. Now we proceed with the details of this outline.

\begin{Assum}\label{assum}Suppose
\begin{enumerate}
  \item $\varphi:(0,\infty)\to (0,\infty)$ is continuous and $\lim\limits_{s\to0^+}\varphi(s)=\infty,$
  \item $\phi(s)=\int_0^s\frac{1}{\varphi(t)}dt$ exists for all $s>0$ and $\lim\limits_{s\to\infty}\phi(s)=\infty.$
\end{enumerate}
\end{Assum}
Suppose $\omega:\mathbb{R}\to [0,1]$ is a smooth function such that
\begin{align*}
\omega(s)=\left\{
  \begin{array}{ll}
    1 & s\geq 2\varepsilon, \\
    0 & s\leq \varepsilon.
  \end{array}
\right.
\end{align*}
Let $\varphi_{\ep},\phi_{\ep}:(0,\infty)\to (0,\infty)$ be defined by
\begin{equation}\label{approx varphi}
\varphi_{\ep}(s)=(1-\omega(s))s^{-n-\ep}+\omega(s)\varphi(s)
\end{equation}
and
\begin{align}\label{xxx}
\phi_{\ep}(s)=\int_0^s\frac{1}{\varphi_{\ep}(t)}dt+\phi(2\ep).
\end{align}
Due to the convexity of $x\mapsto x^{-1}$,
\begin{align}\label{in between 1}
\frac{1}{\varphi_{\ep}(s)}\leq \frac{1}{\varphi(s)}+s^{n+\ep}\quad \forall s\in [0,2\ep].
\end{align}
Moreover, we have
\begin{align}\label{in between 2}
\phi_{\ep}(s)\geq \phi(s)
\end{align}
since by (\ref{xxx}) we have
\begin{align*}
\phi_{\ep}(s)\geq& \phi(2\varepsilon)\geq \phi(s)\quad\quad ~\mbox{for}~ s\leq 2\varepsilon,\\
\phi_{\ep}(s)=&\phi(s)+\int_0^{2\ep}\frac{1}{\varphi_{\ep}(t)}dt\geq \phi(s)\quad\quad ~\mbox{for}~ s\geq 2\varepsilon.
\end{align*}
This last inequality also shows that $\lim_{s\to\infty}\phi_{\ep}(s)=\infty.$
\begin{remark}
Note that
\[\phi(0)=0\quad\mbox{and}\quad \phi_{\ep}(0)=\phi(2\ep).\]
\end{remark}
\begin{lemma}
The functions $\phi_{\ep}$, $1/\varphi_{\ep}$, and $s/\varphi_{\ep}(s)$ converge uniformly in $[0,\infty)$ to $\phi$, $1/\varphi(s)$, and $s/\varphi(s)$ as $\ep\to 0,$ respectively.
\end{lemma}
\begin{proof}
The claims follow from (\ref{approx varphi}), (\ref{in between 1}), (\ref{in between 2}) and
\begin{align}
0\leq \phi_{\ep}(s)-\phi(s)&\leq (2\ep)^{n+1+\ep}+\phi(2\ep)-\phi(\ep)\label{phi-phieps},\\
\Big|\frac{1}{\varphi_{\ep}(s)}-\frac{1}{\varphi(s)}\Big|&\leq \max_{[0,2\ep]}\left(s^{n+\ep}+\frac{2}{\varphi(s)}\right),\label{phi-phieps2}\\
\Big|\frac{s}{\varphi_{\ep}(s)}-\frac{s}{\varphi(s)}\Big|&\leq \max_{[0,2\ep]}\left(s^{n+1+\ep}+\frac{2s}{\varphi(s)}\right).\label{phi-phieps3}
\end{align}
To verify (\ref{phi-phieps}), note that
\begin{align*}
\phi_{\ep}(s)-\phi(s)-\phi(2\varepsilon)=&\int_0^{s}\frac{1}{\varphi_{\varepsilon}(t)}dt-\int_0^{s}\frac{1}{\varphi(t)}dt\\
=&\int_0^{2\varepsilon}\frac{1}{\varphi_{\varepsilon}(t)}dt-\int_0^{2\varepsilon}\frac{1}{\varphi(t)}dt\\
\leq &\int_0^{\varepsilon}t^{n+\ep}dt+\int_{\varepsilon
}^{2\varepsilon}\frac{1}{\varphi(t)}+t^{n+\ep}dt-\int_0^{2\varepsilon}\frac{1}{\varphi(t)}dt\\
=&\int_0^{2\varepsilon}t^{n+\ep}dt-\int_0^{\varepsilon}\frac{1}{\varphi(t)}dt.
\end{align*}
\end{proof}
\begin{lemma}\label{orlicz norm approx}
Let $\mu$ be a finite Borel measure whose support is not contained in a hemisphere. Then for $\ep$ sufficiently small we have
\[\min_{v\in\Sp}\|h_{\hat{v}}\|_{\phi_{\ep},\mu}\geq \frac{1}{2}\min_{v\in\Sp}\|h_{\hat{v}}\|_{\phi,\mu}.\]
\end{lemma}
\begin{proof}
By \cite[Lem.~3.6]{Huang2012}, the values
\[c=\min_{v\in\Sp}\|h_{\hat{v}}\|_{\phi,\mu}~ \mbox{and}~ c_{\ep}=\min_{v\in\Sp}\|h_{\hat{v}}\|_{\phi_{\ep},\mu}~\mbox{are positive.}\]
In view of \cite[Lem.~3]{Haberl2010}, for each $\ep$ there exists $v_{\ep}\in\Sp$ such that
\[\phi_{\ep}(1)=\frac{1}{|\mu|}\int_{\Sp}\phi_{\ep}\left(\frac{h_{\hat{v}_{\ep}}}{c_{\ep}}\right)d\mu.\]
Due to (\ref{in between 2}), $\phi_{\ep}\geq \phi$. Therefore, we have
\begin{align*}
\phi_{\ep}(1)\geq \frac{1}{|\mu|}\int_{\Sp}\phi\left(\frac{h_{\hat{v}_{\ep}}}{c_{\ep}}\right)d\mu.
\end{align*}
Since $\phi$ is increasing, if $c_{\ep_i}\leq \frac{c}{2}$ for a subsequence of $\{\ep_i\}\to0$, then
\begin{align*}
\phi_{\ep_i}(1)&\geq \frac{1}{|\mu|}\int_{\Sp}\phi\left(\frac{2h_{\hat{v}_{\ep_i}}}{c}\right)d\mu.
\end{align*}
Suppose $\lim\limits_{i\to\infty}v_{\ep_i}=w.$ In view of (\ref{phi-phieps}), letting $i\to\infty$ gives
\begin{align*}
\phi(1)&\geq \frac{1}{|\mu|}\int_{\Sp}\phi\left(\frac{2h_{\hat{w}}}{c}\right)d\mu.
\end{align*}
Hence, we arrived at the contradiction $0<c\leq \|h_{\hat{w}}\|_{\phi,\mu}\leq c/2$.
\end{proof}
\begin{lemma}\label{key est}
Let $a,b$ be two positive constants, $\varphi,\varphi_{\ep}$ be defined as above, and
\[\hat{\mathrm{K}}_o:=\{K\in \mathrm{K}_o:  w_-(K)\geq a,~ h_K\leq b \}.\] Suppose $\mu,\{\mu_i\}_{i\in \mathbb{N}}$ are finite Borel measures on $\Sp$ such that $\mu_i\to \mu$ weakly and the support of $\mu$ is not contained in a closed hemisphere. Then there exist positive constants $\lambda,\Lambda$ depending on $a,b,\varphi,\mu$ such that for all $K\in \hat{\mathrm{K}}_o$, $i$ sufficiently large and $\ep$ small enough,
\begin{align*}
\lambda\leq \int_{\Sp}\frac{h_K}{\varphi(h_K)}d\mu_i\leq \Lambda,\quad
\lambda\leq \int_{\Sp}\frac{h_K}{\varphi_{\ep}(h_K)}d\mu_i\leq \Lambda.
\end{align*}
\end{lemma}
\begin{proof}
Recall that $s/\varphi(s)$ is uniformly continuous on $[0,b]$.

First, we show each $\mu_i$ satisfies
\[\int_{\Sp}\frac{h_K}{\varphi(h_K)}d\mu_i\geq c_i\quad \forall K\in \hat{\mathrm{K}}_o,\]
for some constant $c_i>0.$
If the lower bound was zero, then by the Blaschke selection theorem we could find a sequence of convex bodies converging to a convex body $\hat{K}\in\hat{\mathrm{K}}_o$ with $\int_{\Sp}\frac{h_{\hat{K}}}{\varphi(h_{\hat{K}})}d\mu_i=0$. Thus, we would have $\mu_i(\{h_{\hat{K}}>0\})=0$. But the set $\{h_{\hat{K}}>0\}$ contains an open hemisphere, and for large $i$, $\mu_i(\Sigma)>0$ for any open hemisphere $\Sigma$.

Second, we show that $\{c_i\}$ is uniformly bounded below away from zero. Otherwise, we can find a sequence $\{K_i\}\subset \hat{\mathrm{K}}_o$ such that
\[\int_{\Sp}\frac{h_{K_i}}{\varphi(h_{K_i})}d\mu_i\to 0\quad\hbox{as}~i\to\infty.\]
Suppose $K_i\to K\in \hat{\mathrm{K}}_o$ as $i\to\infty$. We have
\begin{align*}
\int_{\Sp}\frac{h_{K_i}}{\varphi(h_{K_i})}d\mu_i-\int_{\Sp}\frac{h_{K}}{\varphi(h_{K})}d\mu
=&\int_{\Sp}\left(\frac{h_{K_i}}{\varphi(h_{K_i})}-
\frac{h_{K}}{\varphi(h_{K})}\right)d\mu_i\\
&+\int_{\Sp}\frac{h_{K}}{\varphi(h_{K})}d\mu_i-\int_{\Sp}\frac{h_{K}}{\varphi(h_{K})}d\mu.
\end{align*}
In view of (\ref{phi-phieps3}), each line of the right-hand side goes to zero thus yielding a contradiction.

Finally, due to (\ref{phi-phieps3}), for every $\delta>0$, there exist $\ep_0,N$, such that for all $\ep\leq \ep_0$, $i\geq N$ and $K\in \hat{\mathrm{K}}_o,$
we have
\begin{align*}
\int_{\Sp}\frac{h_K}{\varphi_{\ep}(h_K)}d\mu_i=&\int_{\Sp}\frac{h_K}{\varphi(h_K)}d\mu_i
+\int_{\Sp}\left(\frac{h_K}{\varphi_{\ep}(h_K)}-\frac{h_K}{\varphi(h_K)}\right)d\mu_i\\
\geq& \int_{\Sp}\frac{h_K}{\varphi(h_K)}d\mu_i-\delta |\mu|\\
\geq& c_i-\delta |\mu|.
\end{align*}
Therefore, the uniform lower bound follows by taking $\delta$ small enough.

Now we prove the upper bounds. For $i$ sufficiently large, we have
\[\int_{\Sp}\frac{h_K}{\varphi(h_K)}d\mu_i\leq |\mu_i|\max_{s\in [0,b]}\frac{s}{\varphi(s)}\leq 2|\mu|\max_{s\in [0,b]}\frac{s}{\varphi(s)}.\]

Moreover, due to (\ref{phi-phieps3}), there exist $\ep_0,N$, such that for all $\ep\leq \ep_0$, $i\geq N$ and $K\in \hat{\mathrm{K}}_o,$ we have
\begin{align*}
\int_{\Sp}\frac{h_K}{\varphi_{\ep}(h_K)}d\mu_i\leq& \int_{\Sp}\frac{h_K}{\varphi(h_K)}d\mu_i+|\mu|.
\end{align*}
\end{proof}

\begin{proposition}\label{prop}
Suppose $\varphi:(0,\infty)\to (0,\infty)$ is a smooth such that Assumption \ref{assum} is satisfied.
Let $f:\Sp\to (0,\infty)$ be smooth and $\ep$ be sufficiently small.
There exists a closed, smooth, strictly convex hypersurface $M_{\ep}$ with positive support function such that
\begin{align}\label{approx}
f\varphi_{\ep}(h)\sigma_n=\gamma_{\ep},\quad\hbox{where}\quad \gamma_{\ep}=\frac{\int_{\Sp} h\sigma_nd\theta}{\int_{\Sp}\frac{h}{\varphi_{\ep}(h)}d\mu_f}.
\end{align}
In addition, we have
\[w_+(M_{\ep})\leq \frac{4}{\min\limits_{v\in\Sp}\|h_{\hat{v}}\|_{\phi,\mu_f}},\quad w_-(M_{\ep})\geq \left(\frac{\min\limits_{v\in\Sp}\|h_{\hat{v}}\|_{\phi,\mu_f}}{4}\right)^nV(\mathbb{S}^n).\]
Moreover, for some positive constants $\lambda,\Lambda$ depending on $\varphi,\mu_f$,
\[\lambda\leq \gamma_{\ep}\leq \Lambda.\]
Also, if $f$ is invariant under a closed group $G\subset O(n+1),$ then $M_{\ep}$ can be chosen to be $G$-invariant.
\end{proposition}
\begin{proof}
Let $M_0$ be the unit sphere. We seek a family of smooth, strictly convex hypersurfaces $\{M_{t,\ep}\}$ whose support functions satisfy
\begin{align}\label{nf-approx}
\left\{
  \begin{array}{ll}
\partial_th= f  h\varphi_{\ep}(h)\sigma_n-\frac{\int_{\Sp} h\sigma_nd\theta}{\int_{\Sp}\frac{h}{\varphi_{\ep}(h)}d\mu_f} h;\\
h(\cdot,0)\equiv1.
  \end{array}
\right.
\end{align}
We prove the flow hypersurfaces subconverge smoothly to a solution of (\ref{approx}) with the desired properties. To do this, all we need is to obtain $C_0$-estimates; the higher order regularity estimates and convergence follow as in Sections \ref{reg est} and \ref{sec:conv}. Moreover, the statement regarding $G$-invariance follows, since the flow hypersurfaces are $G$-invariant.

Suppose $T_{\star}$ is the maximal time that the flow hypersurfaces are smooth, strictly convex and contain the origin in their interiors. First, we obtain a uniform upper bound for support functions on $[0,T_{\star})$ independent of $T_{\star}$. Then using this bound, we establish a uniform positive lower bound on the support functions on $[0,T_{\star})$ independent of $T_{\star}$. Thus the flow hypersurfaces will always contain the origin in their interiors as long as they are smooth and strictly convex.

Let us put $h_{t,\ep}=h_{M_{t,\ep}}$ and define
\[\mathcal{E}(t)=\frac{1}{|\mu_f|}\int_{\Sp}\phi_{\ep}(h_{t,\ep})d\mu_f.\]
Since $\frac{d}{dt}\mathcal{E}(t)=0,$ we have $\mathcal{E}(t)=\mathcal{E}(0)=\phi_{\ep}(1).$
Thus from the definition of the Orlicz norm it follows that
\begin{align*}
\|h_{t,\ep}\|_{\phi_{\ep},\mu_f}\leq 1.
\end{align*}
Choose $v\in\Sp$ and $R>0$, such that $Rv\in M_{t,\ep}$ and $|Rv|$ is maximal. The line segment joining the origin and $Rv$ is the largest such line segment contained in the enclosed region by $M_{t,\varepsilon}$ and its support function is $Rh_{\hat v}$. Therefore, by this inclusion $h_{t,\varepsilon}\geq Rh_{\hat v}. $
Due to  (\ref{basic prop orlicz}),
\begin{equation}\label{R}
R\min_{v\in\Sp}\|h_{\hat{v}}\|_{\phi_{\ep},\mu_f}\leq R\|h_{\hat{v}}\|_{\phi_{\ep},\mu_f}\leq \|h_{t,\ep}\|_{\phi_{\ep},\mu_f}\leq 1.
\end{equation}
By Lemma \ref{orlicz norm approx}, for $\ep$ sufficiently small, we have
\[\min_{v\in\Sp}\|h_{\hat{v}}\|_{\phi_{\ep},\mu_f}\geq \frac{1}{2}\min_{v\in\Sp}\|h_{\hat{v}}\|_{\phi,\mu_f}.\]
Together with \eqref{R}, this last inequality yields an upper bound on $R=\max h_{t,\ep}$ and the maximum width:
\[w_+(M_{t,\varepsilon})\leq 2R\leq \frac{4}{\min\limits_{v\in\Sp}\|h_{\hat{v}}\|_{\phi,\mu_f}}.\] Moreover, since the volume is non-decreasing along the flow, the lower bound on the minimum width follows as well:
\begin{align*}
V(\mathbb{S}^n)&\leq V(M_{t,\varepsilon})\leq w_-(M_{t,\varepsilon})w_+(M_{t,\varepsilon})^{n}\\
&\leq \left(\frac{4}{\min\limits_{v\in\Sp}\|h_{\hat{v}}\|_{\phi,\mu_f}}\right)^nw_-(M_{t,\varepsilon}).
\end{align*}
Thus by a special case of Lemma \ref{key est} ($\mu=\mu_i=d\mu_f$) we obtain
\[\lambda\leq\frac{\int_{\Sp} h\sigma_nd\theta}{\int_{\Sp} \frac{h}{\varphi_{\ep}(h)}d\mu_f}\leq \Lambda\]
for some positive constants. Therefore, as soon as $h_{\min}\leq \ep$,
\begin{align*}
\partial_t\log h_{\min}\geq& \varphi_{\ep}(h_{\min})h_{\min}^nf_{\min}-\Lambda\\
            =&h_{\min}^{-\ep}f_{\min}-\Lambda.
\end{align*}
Thus for each $\ep$, $h_{t,\ep}$ is uniformly bounded below away from zero.
\end{proof}
\begin{corollary}\label{cor}
Suppose $\varphi:(0,\infty)\to (0,\infty)$ is a smooth function such that Assumption \ref{assum} is satisfied.
Let $f:\Sp\to (0,\infty)$ be smooth. Then there exists
 a convex body $K\in \mathrm{K}_o$ such that
\begin{align*}
 dS_K= \frac{\gamma}{\varphi(h_K)}d\mu_f\quad \mbox{for a constant}~\gamma>0.
\end{align*}
In addition, we have
\[w_+(K)\leq \frac{4}{\min\limits_{v\in\Sp}\|h_{\hat{v}}\|_{\phi,\mu_f}},\quad w_-(K)\geq \left(\frac{\min\limits_{v\in\Sp}\|h_{\hat{v}}\|_{\phi,\mu_f}}{4}\right)^nV(\Sp),\]
and for some positive constants depending only on $\varphi,\mu_f$, \[\lambda\leq\gamma\leq \Lambda.\]
Moreover, if $f$ is invariant under a closed group $G\subset O(n+1),$ then $K$ can be chosen to be invariant under $G.$
\end{corollary}
\begin{proof}
Let $\ep=\frac{1}{i}$ in Proposition \ref{prop} and $\varphi_i=\varphi_{\ep}$. For each $i$, we have a solution $K_i$ containing the origin in its interior such that
\[dS_{K_i}=\frac{\gamma_i}{\varphi_{i}(h_{K_i})}d\mu_f.\]
Moreover, $\{h_{K_i}\}$ is uniformly bounded above, we have a uniform positive lower bound on $w_-(K_i)$, and $\gamma_i\in [\lambda,\Lambda]$. Thus by the Blaschke selection theorem, we may find a subsequence of $\{K_i\}$ converging to a limit $K\in \mathrm{K}_o$ while $\gamma_i\to\gamma\in [\lambda,\Lambda]$. Note that if $f$ is invariant under $G$, then all $K_i$ and thus $K$ can be chosen to be invariant under $G$.

Suppose $\psi\in C(\Sp)$. Note that
\begin{align*}
\int_{\Sp}\psi\left(\frac{1}{\varphi_{i}(h_{K_i})}-\frac{1}{\varphi(h_{K})}\right)d\mu_f
=&\int_{\Sp}\psi\left(\frac{1}{\varphi_{i}(h_{K_i})}-\frac{1}{\varphi(h_{K_i})}\right)d\mu_f\\
&+\int_{\Sp}\psi\left(\frac{1}{\varphi(h_{K_i})}-\frac{1}{\varphi(h_{K})}\right)d\mu_f,
\end{align*}
In view of (\ref{phi-phieps2}), the first line on the right-hand side tends to zero as $i\to\infty.$ Since $1/\varphi$ is continuous and $h_{K_i}$ converges to $h_{K}$ uniformly, the term on the second line converges to zero as well. Since $S_{K_i}\to S_K$ weakly (cf., \cite{Schneider2013a}), $K$ satisfies
\[\int_{\Sp}\psi dS_{K}=\gamma\int_{\Sp}\frac{\psi}{\varphi(h_{K})}d\mu_f\quad \forall \psi\in C(\Sp).\]
\end{proof}
\subsection{Proofs}
\begin{proof}[Proof of Theorem \ref{main theoremGeneral}]
For the moment, we suppose $\varphi$ is smooth and Assumption \ref{assum} is satisfied. Using approximation by discrete measures \cite[Thm. 8.2.2]{Schneider2013a} and then bump functions in turn \cite[Lem. 3.7]{Haodi2019}, there exists a family of positive functions $\{f_i\}\subset C^{\infty}(\Sp)$ such that
$\mu_{f_i}\to\mu$ weakly as $i\to\infty.$ We can choose $\{f_i\}$ to be invariant under $G$ when $\mu$ is $G$ invariant. Since the support of $\mu$ is not contained in a closed hemisphere, by \cite[Cor.~3.7]{Huang2012} we obtain
\begin{align}\label{ww}
\min_{v\in\Sp}\|h_{\hat{v}}\|_{\phi,\mu_{f_i}}\geq\delta>0.
\end{align}

Suppose $K_i$ is a solution corresponding to the measure $\mu_{f_i}$ given by Corollary \ref{cor}. By (\ref{ww}), we have uniform lower and upper bounds on $w_-(K_i)$ and $w_+(K_i).$ Thus a subsequence of $\{K_i\}$ converges to a convex body $K\in \mathrm{K}_o$.

Since $\mu_{f_i}\to \mu$ weakly, by Lemma \ref{key est} the constants $\lambda,\Lambda$ of Corollary \ref{cor} depend only on $w_-,w_+,\varphi,\mu.$ Hence, $K$ is the desired solution. A further approximation allows us to assume $\varphi$ is merely continuous.
\end{proof}
\begin{proof}[Proof of Theorem \ref{main theoremHLYZ}]
We only mention the necessary changes in the argument leading to Theorem \ref{main theoremGeneral}. Suppose $\{f_i\}\subset C^{\infty}(\Sp)$ is a family of positive even functions such that
$\mu_{f_i}\to\mu$ weakly as $i\to\infty.$
We consider the flow of smooth, strictly convex hypersurfaces $\{M_{t,i}\}$ whose support functions satisfy
\begin{align*}
\left\{
  \begin{array}{ll}
\partial_th= f_i  h\varphi(h)\sigma_n-\frac{\int_{\Sp} h\sigma_nd\theta}{\int_{\Sp}\frac{h}{\varphi(h)}d\mu_{f_i}} h;\\
h(\cdot,0)\equiv1.
  \end{array}
\right.
\end{align*}
Note for each $i$, the flow hypersurfaces $\{M_{t,i}\}$ remain $o$-symmetric and converge smoothly to a strictly convex hypersurface $M_{i}$, enclosing the region $K_i$, which solves
\[\varphi(h_{K_i})dS_{K_i}=\frac{(n+1)V(M_i)}{\int_{\Sp}\frac{h_{K_i}}{\varphi(h_{K_i})}d\mu_{f_i}}d\mu_{f_i}.\]

For $i$ sufficiently large, the uniform upper bound for $\{h_{K_i}\}$ follows from the following argument. For $v\in\mathbb{S}^n$, let $h_{\bar{v}}$ be the support function of the line segment joining $\pm v$. Then we have $\min_{v\in\mathbb{S}^n}\|h_{\bar{v}}\|_{\phi,\mu_{f_i}}>0.$ Now for any $v$ and $R$ such that $\pm Rv\in M_{i}$ with maximal distance from the origin, we have $Rh_{\bar{v}}\leq h_{K_i}$. Therefore, arguing similarly as in (\ref{R}), we deduce \[R\min_{v\in\mathbb{S}^n}\|h_{\bar{v}}\|_{\phi,\mu_{f_i}}\leq 1.\] 

Since $\mu$ is a finite even Borel measure which is not concentrated on a great subsphere, by \cite[Cor.~3.7]{Huang2012} we have
\begin{align*}
\min_{v\in\Sp}\|h_{\bar{v}}\|_{\phi,\mu_{f_i}}\geq\delta>0 \quad \mbox{for $i$ large enough}.
\end{align*}
Hence due to $V(M_i)\geq V(\mathbb{S}^n),$ we find constants $a,b$ such that
\[a\leq h_{K_i}\leq b.\]
This yields $\lambda, \Lambda$, depending only on $a,b$, such that
\begin{align*}
\lambda\leq \int_{\Sp}\frac{h_{K_i}}{\varphi(h_{K_i})}d\mu_{f_i}\leq \Lambda \quad \mbox{for $i$ sufficiently large}.
\end{align*}
Therefore, a subsequence of $\{K_i\}$ converges to a solution of (\ref{measure orlicz prob}).
\end{proof}
\begin{remark}
Note that in Theorem \ref{main theoremHLYZ}, the condition $\lim_{s\downarrow 0}\varphi(s)=\infty$ is relaxed. In this case, the lower bound on the support function follows easily as is explained in the proof of Theorem \ref{main theoremHLYZ}. While to prove Theorem \ref{main theoremGeneral}, we use an approximation argument in which it is essential that the right-hand side of (\ref{phi-phieps2}) converges to zero; cf., Corollary \ref{cor}.
\end{remark}
\section{Christoffel-Minkowski type problems}
Given a smooth, positive function $f$ defined on the unit sphere, the $L_p$-Christoffel--Minkowski problem asks for a smooth, strictly convex hypersurface with positive support function $h$ that satisfies
\[fh^{1-p}\sigma_k=\gamma, \quad \mbox{for some constant}~\gamma>0.\]
Here $p\in \mathbb{R}$, $\sigma_k$ is the $k$-th elementary symmetric polynomial
\[\sigma_{k}(\lambda)=\sum_{1\leq i_{1}<\dots<i_{k}\leq n}\lambda_{i_{1}}\cdots \lambda_{i_{k}},\]
and $\lambda_i$ are the principal radii of curvature.
The reader may consult \cite{Guan2018a,Guan2003a,Hu2004} and references therein regarding this problem.

In this section, we solve a class of $L_p$-Christoffel--Minkowski type problems, allowing a broader class of curvature functions in place of $\sigma_k,$ which can be considered as a complement of \cite[Thm. 3.13]{BIS2020}. Our proof remains brief and only the necessary modifications of the proof in \cite{Ivaki2019} are highlighted. Define
\[\Gamma_+=\{(\lambda_i)=(\lambda_1,\dots,\lambda_n)\in\mathbb{R}^n; \lambda_i>0\}.\]

\begin{Assum}\label{assum+generalization} Suppose
$F\in C^{\infty}(\Gamma_+)$ is positive, 1-homogeneous, strictly monotone, $F(1,\dots,1)=1,$ inverse concave, i.e., \[F_{\ast}(\lambda_i):=\frac{1}{F(\lambda_i^{-1})}\quad \mbox{is concave,}\]
  and $F_{\ast}|_{\partial\Gamma_+}=0$.
\end{Assum}
By \cite[Lem. 2.2.11]{Gerhardt2006b}, $\sigma_k^{\frac{1}{k}}$ satisfies Assumption \ref{assum+generalization}.
\begin{theorem}
Let $k\in \mathbb{N},$ $p>k+1$ and $0<f\in C^{\infty}(\mathbb{S}^n)$ satisfy
\[\bar{\nabla}_i\bar{\nabla}_jf^{\frac{1}{p+k-1}}+\bar{g}_{ij}f^{\frac{1}{p+k-1}}>0.\]
Suppose $F$ is a function of the principal radii of curvature and satisfies Assumption \ref{assum+generalization}.
Then there exists a closed, smooth, strictly convex hypersurface with positive support function that solves
\[fh^{1-p}F^k=1.\]
\end{theorem}
To prove this theorem, we use a curvature flow similar to (\ref{F-param2}).
For a suitable smooth, strictly convex hypersurface $M_0$ parameterized by
\[x_0: M\to\mathbb{R}^{n+1},\]
we seek smooth, strictly convex hypersurfaces $M_t$ satisfying
\begin{align}\label{F-param4}
\left\{
  \begin{array}{ll}
  \partial_tx= f (\nu)\frac{\langle x,\nu\rangle^{2-p}}{F_{\ast}^k}\nu- x; \\
   x(\cdot,0)=x_0(\cdot).
  \end{array}
\right.
\end{align}
Therefore, the support functions of the $M_t$ satisfy
\begin{align}
\partial_th=fh^{2-p}F^k-h.
\end{align}

Compared to the flow \cite[(1.4)]{Ivaki2019}, here we do not have the global term $\eta$. Moreover, since in general the problem $fh^{1-p}F^k=1$ does not have any variational structure, we cannot expect the convergence of solutions for all initial hypersurfaces. Therefore, we follow the approach of \cite{BIS2020,Gerhardt2006b} to obtain the regularity estimates.

Step 1: Choose $M_0$ to be a sufficiently small origin-centered sphere such that \[fh^{1-p}F^k|_{M_0}>1.\] This is possible due to $p>k+1.$

Step 2: We obtain uniform lower and upper bounds on the support function as in the proof of Lemma \ref{C0 estimate} here (the case of Theorem \ref{main theorem2}).

Step 3: Let us put
\[G=F^k,\quad \Theta=f h^{2-p},\quad \mathcal{L}=\partial_t-fh^{2-p}G^{ij}\bar{\nabla}_i\bar{\nabla}_j.\]
We have
\begin{align}\label{soliton eq2}
\mathcal{L}(\frac{\Theta G}{h}-1)=&2\Theta G^{ij}\n_i\log h\n_j\left(\frac{\Theta G}{h}\right)
\nonumber
\\&+\left(1+k-p\right)\left(\frac{\Theta G}{h}\right)\left(\frac{\Theta G}{h}-1\right).
\end{align}
Thus, $fh^{1-p}F^k$ is uniformly bounded above and below. By Step 2, we obtain uniform lower and upper bounds on $F.$

Step 4: Let $r^{ij}$ denote the inverse of principal radii of curvature. Since $F$ is inverse concave, as in \cite[Lem. 2.7]{Ivaki2019}, we can apply the maximum principle to the auxiliary function
\[\frac{r^{ij}}{h}\]
to obtain a uniform lower bound on the principal radii of curvature.

Step 5: Let us arrange the principal radii of curvature as \[\lambda_1\leq \cdots\leq \lambda_n.\] Due to Steps 3 and 4,
\[C\geq\frac{F}{\lambda_1}=F(1,\ldots,\frac{\lambda_n}{\lambda_1})=\frac{1}{F_{\ast}(\frac{\lambda_1}{\lambda_n},\ldots,1)},\]
for some constant $C.$ Hence, by Assumption \ref{assum+generalization}, we get a uniform upper bound on the principal radii of curvature.

Step 6: Now that we have uniform $C^2$-estimates, higher order regularity estimates follow as well and the flow smoothly exists on $[0,\infty).$

Step 7: By (\ref{soliton eq2}), we have
\[\partial_th=fh^{2-p}F^k-h>0\quad \mbox{for all}~t\geq 0.\]
Moreover, due to Step 2,
\begin{align*}
0<h(u,t)-h(u,0)=\int_0^t (fh^{2-p}F^k-h)(u,s)ds<\infty.
\end{align*}
Therefore, in view of monotonicity of $h$, the limit
\[\tilde{h}(u):=\lim_{t\to\infty}h(u,t)\]
exists and is positive, smooth and $\bar{\nabla}^2\tilde{h}+\bar{g}\tilde{h}>0$. Thus the hypersurface with support function $\tilde{h}$ is our desired solution to
\[fh^{1-p}F^k=1.\]
\section*{Acknowledgment}
PB was supported by the ARC within the research grant ``Analysis of fully non-linear geometric problems and differential equations", number DE180100110. MI was supported by a Jerrold E. Marsden postdoctoral fellowship from the Fields Institute. JS was supported by the ``Deutsche Forschungsgemeinschaft" (DFG, German research foundation) within the research scholarship ``Quermassintegral preserving local curvature flows", grant number SCHE 1879/3-1.
\bibliographystyle{amsalpha-nobysame}
\bibliography{library}
\vspace{10mm}

\textsc{Department of Mathematics, Macquarie University,\\ NSW 2109, Australia, }\email{\href{mailto:paul.bryan@mq.edu.au}{paul.bryan@mq.edu.au}}

\vspace{2mm}

\textsc{Department of Mathematics, University of Toronto,\\ Ontario, M5S 2E4, Canada, }\email{\href{mailto:m.ivaki@utoronto.ca}{m.ivaki@utoronto.ca}}

\vspace{2mm}

\textsc{Department of Mathematics, Columbia University,\\ New York, NY 10027, USA, }\email{\href{mailto:jss2291@columbia.edu}{jss2291@columbia.edu}}
\end{document}